\documentclass[letterpaper, 10 pt, conference]{ieeeconf}  

\IEEEoverridecommandlockouts                              

\title{\LARGE \bf
Efficient Input-Constrained Impulsive Optimal Control of Linear Systems with Application to Spacecraft Relative Motion
}

\author{Ethan Foss$^{1}$ and Simone D'Amico$^{1}$
\thanks{$^{1}$Department of Aeronautics and Astronautics, Stanford University, 94305, USA, 
        {\tt\small \{erfoss,damicos\}@stanford.edu}}%
}

\usepackage{cite}
\usepackage{amsmath,amssymb,amsfonts}

\newtheorem{theorem}{Theorem}[section]
\newtheorem{lemma}{Lemma}

\newtheorem{remark}{Remark}
\newtheorem{problem}{Problem}

\usepackage[caption=false,font=footnotesize]{subfig}

\DeclareMathOperator*{\arglocmax}{arg\,loc\,max}
\usepackage{algorithm}
\usepackage{algorithmic}
\makeatletter
\newcommand\fs@norules{\def\@fs@cfont{\bfseries}\let\@fs@capt\floatc@ruled
  \def\@fs@pre{}%
  \def\@fs@post{}%
  \def\@fs@mid{\kern3pt}%
  \let\@fs@iftopcapt\iftrue}
\makeatother
\floatstyle{norules}
\restylefloat{algorithm}
\usepackage{graphicx}
\usepackage{textcomp}
\usepackage{xcolor}
\usepackage{bm}

\begin{document}

\maketitle
\thispagestyle{empty}
\pagestyle{empty}

\begin{abstract}

This work presents a novel algorithm for impulsive optimal control of linear time-varying systems with the inclusion of input magnitude constraints. Impulsive optimal control problems, where the optimal input solution is a sum of delta functions, are typically formulated as an optimization over a normed function space subject to integral equality constraints and can be efficiently solved for linear time-varying systems in their dual formulation. In this dual setting, the problem takes the form of a semi-infinite program which is readily solvable in online scenarios for constructing maneuver plans. This work augments the approach with the inclusion of magnitude constraints on the input over time windows of interest, which is shown to preserve the impulsive nature of the optimal solution and enable efficient solution procedures via semi-infinite programming. The resulting algorithm is demonstrated on the highly relevant problem of relative motion control of spacecraft in Low Earth Orbit (LEO).

\end{abstract}

\section{Introduction}
The problem of minimizing the norm of a control input function space of a linear time-varying system with boundary conditions has received significant attention due to its practical application to many systems of interest. Such optimal control problems arise in many fields and are central to the study of spacecraft optimal control and trajectory optimization.
Despite its widespread applicability, impulsive optimal control remains challenging for many applications of interest, which may involve complex state and input constraints and require efficient and reliable solution procedures. For example, in the context of spacecraft relative motion control, for which computational resources are limited, safety is essential, and algorithmic reliability is crucial, few existing approaches are adequate. 

Approaches to solving impulsive optimal control problems are numerous, and include direct methods, such as sequential convex programming \cite{mao2019successiveconvexificationsuperlinearlyconvergent}, indirect methods, leveraging moment problems \cite{neustadt1964}, reachable set theory \cite{gilber1971}, or primer-vector theory \cite{prussing}, and other techniques, such as analytical methods \cite{chernick}. In direct methods, the dynamics and cost are discretized over a time grid of interest and solved as a finite-dimensional optimization problem. Despite widespread utility, these techniques do not explicitly exploit solution sparsity and therefore tend to introduce many optimization variables or can produce suboptimal solutions in the case of coarse gridding. 


Indirect methods, in comparison, involve formulating optimality conditions on the continuous-time problem, and include moment-based formulations \cite{neustadt1964}, reachable set theory \cite{gilber1971}, and primer-vector theory approaches \cite{lawden1963optimal}, which provide different perspectives to the problem and yet all arrive at equivalent optimality conditions. 

Early algorithmic frameworks for leveraging the optimality conditions derived from these indirect methods are proposed in Gilbert \cite{gilber1971}, which employed an early numerical algorithm based on quadratic programming \cite{barr} to efficiently solve impulsive rendezvous problems. Other approaches involve simplified techniques to solve the necessary conditions to suboptimality \cite{Carter1995} or applying mixed iterative schemes to solving the polynomial necessary and sufficient conditions of optimality \cite{Arzeliermixediterative}. In 2013, Arzelier \cite{ARZELIER2016373} connected the moment problem theorized by Neustadt \cite{neustadt1964} to a semi-infinite programming convex optimization scheme leveraging exchange methods \cite{Reemtsen1998} to produce an efficient, convergent, and reliable strategy to solving the linearized impulsive optimal control problem. This semi-infinite programming approach has served as the foundation for subsequent impulsive optimal control studies \cite{koenig2019,hunter2024} and has heritage on CubeSat missions such as VISORS \cite{guffanti2023autonomousguidancenavigationcontrol}.

Despite the practicality of these theoretical and algorithmic frameworks, the applicability of indirect methods to constrained problems remains limited. With respect to spacecraft rendezvous, a major constraint of interest involves bounding the maximum impulse that a spacecraft can execute at a given time. All satellites, such as those of the VISORS mission, have propulsion systems that naturally constrain the amount of $\Delta V$ that can be applied at a given time, necessitating the use of maneuver splitting algorithms or other heuristic approaches \cite{guffanti2023autonomousguidancenavigationcontrol} which lead to suboptimal maneuver plans or increased computational burden. More broadly, it is often practical to distribute $\Delta V$ expenditure for spacecraft relative motion irrespective of propulsion constraints to reduce the aggressiveness of maneuver plans. 

To address this theoretical and algorithmic gap, this paper introduces impulse magnitude constraints by applying inequality constraints to the integral of the input norm over specified time windows. It is shown that the resulting formulation preserves the impulsive quality of the solution and can be easily extended to semi-infinite programming techniques to yield an efficient and reliable solution procedure. The algorithm is demonstrated for several optimal control problems of interest in spacecraft relative motion control. In contrast to other formulations of input-constrained optimal control \cite{acikmese,taheri}, which are ubiquitous in aerospace literature, the formulation presented here is unique in that it preserves solution impulsivity, allowing for both improved algorithmic procedures and more desirable maneuver plans.

\section{Problem Background}

This work considers linear, time-varying systems of the form 
\begin{equation}
    \dot{\bm{x}}(t) = A(t) \bm{x}(t) + B(t) \bm{u}(t),
\end{equation}
where $\bm{x}: [t_0,t_f] \to \mathbb{R}^{n_x}, t \to \bm{x}(t)$ is the state trajectory represents the system state, $\bm{u}: [t_0,t_f] \to \mathbb{R}^{n_u}, t \to \bm{u}(t)$ is the input signal, $A: [t_0,t_f] \to \mathbb{R}^{n_x \times n_x}$ is the piecewise-continuous plant matrix, and $B: [t_0,t_f] \to \mathbb{R}^{n_x \times n_u}$ is the input matrix. Boundary conditions are enforced at fixed initial and terminal times such that
\begin{equation}
    \bm{x}(t_0) = \bm{x}_0, \ \bm{x}(t_f) = \bm{x}_f,
\end{equation}
where $\bm{x}_0 \in \mathbb{R}^{n_x}$ and $\bm{x}_f \in \mathbb{R}^{n_x}$ are fixed initial and terminal constraints. The set of input functions that satisfy the system dynamics and boundary conditions can be expressed as an equality constraint in terms of the state transition matrix via
\begin{equation}
    \bm{x}_f = \Phi(t_f,t_0) \bm{x}_0 + \int_{t_0}^{t_f} \Phi(t_f,\tau) B(\tau) \bm{u}(\tau) d\tau. \label{eq:BasicDynamics}
\end{equation}
Impulsive optimal control problems typically seek to minimize an accumulated norm of the control input, i.e.,
\begin{equation}
    J(\bm{u}(\cdot)) = \int_{t_0}^{t_f} f(\bm{u}(\tau)) d\tau, \label{eq:BasicCost}
\end{equation}
where $f: \mathbb{R}^{n_u} \to \mathbb{R}^+$ is a norm
that satisfies positive definiteness ($f(\bm{u}) \ge 0 \ \forall \bm{u} \in \mathbb{R}^{n_u}$, homogeneity ($f(\alpha \bm{u}) = \alpha f(\bm{u}) \ \forall \alpha \in \mathbb{R}^+$), and the triangle inequality ($f(\bm{u}_1) + f(\bm{u}_2) \ge f(\bm{u}_1+\bm{u}_2) \ \forall \bm{u}_1,\bm{u}_2 \in \mathbb{R}^{n_u}$).

\begin{problem}[Basic Impulsive Control]
    Find $\bm{u}(\cdot)$ to minimize $J(\bm{u}(\cdot))$ (Equation~\ref{eq:BasicCost}) such that the dynamics (Equation~\ref{eq:BasicDynamics}) are satisfied. \label{prob:BasicProblem}
\end{problem}

Problem~\ref{prob:BasicProblem} is a moment problem of the same form given in Neustadt \cite{neustadt1964}, whose optimal solution is a sum of delta functions, i.e. the input is an atomic measure. In its primal form, this problem is difficult to solve due to its infinite dimensional nature.
However, in the dual setting, the problem is conveniently reformulated as the following convex semi-infinite program
\begin{problem}[Basic Dual Problem]
\begin{equation}
    \begin{aligned}
\sup_{\bm{\lambda}_f} \quad & \bm{\lambda}_f^\intercal (\bm{x}_f - \Phi(t_f,t_0) \bm{x}_0) \\
\text{s.t.} \quad & f^\circ(B^\intercal(t) \Phi^\intercal(t_f,t) \bm{\lambda}_f ) \le 1 \quad \forall t \in [t_0,t_f]
\end{aligned}
\end{equation}
\label{prob:BasicDual}
\end{problem}
where the optimization variable $\bm{\lambda}_f$ is a Lagrange multiplier and $f^\circ(\cdot): \mathbb{R}^{n_u} \to \mathbb{R}^+$ is the dual norm of $f$, i.e. $f^\circ(\bm{y}) = \sup_{f(\bm{x})\le 1} \bm{x}^\intercal \bm{y}$. For many relevant norms, e.g., $l_p$ norms, the dual norm is analytically available ($\| \cdot \|_p^{\circ} = \| \cdot \|_q, 1/p+1/q=1$, see also \cite{koenig2019}).

In contrast to the infinite dimensional primal problem, the dual problem is an optimization problem over a finite number of variables subject to an infinite constraint space. The dual form of this problem is in fact easier to solve in this form, as Arzelier \cite{ARZELIER2016373} leveraged, through exchange methods for semi-infinite programs \cite{Reemtsen1998,Hettich2009}.

The moment problem is connected to both primer-vector theory and reachable set theory approaches through its dual form. Indeed, $\bm{\lambda}_v(t) \triangleq B^\intercal(t) \Phi^\intercal (t_f,t) \bm{\lambda}_f$ is in fact the classic primer-vector \cite{lawden1963optimal}, with the Lagrange multiplier $\bm{\lambda}_f$ being the costate at the terminal time, and the semi-infinite constraint in Problem~\ref{prob:BasicDual} is equivalent to that derived in Prussing \cite{prussing}. From the perspective of reachable set theory \cite{gilber1971}, the dual norm of the costate, $f^\circ(B^\intercal(t)\Phi^\intercal(t_f,t)\bm{\lambda}_f)$, is simply the support function of the sublevel set of the unit cost at time $t$, and $\max_{t\in [t_0,t_f]} f^\circ(B^\intercal(t)\Phi^\intercal(t_f,t)\bm{\lambda}_f)$ is nothing but the support function of the pseudostate reachable set under unit cost \cite{koenig2019}. 

\begin{problem}[Basic Maneuver Reconstruction]
    Let $\bm{\lambda}_f^*$ be the optimal solution to Problem~\ref{prob:BasicDual} and $t^* = \{t \in [t_0,t_f] \mid f^\circ(B^\intercal(t) \Phi^\intercal(t_f,t) \bm{\lambda}_f^* ) = 1\}$. The set $t^*$ represents the optimal impulse times to the system. Neustadt \cite{neustadt1964} showed that the number of impulse times is at most $n_x$, and that selecting $n_x$ components of $t^*$ recovers the optimal semi-infinite variables. From these times, the optimal control can be constructed as $\bm{u}(t) = \sum_{t_k \in t^*} \bm{v}_k \delta(t-t_k)$ where $\{ \bm{v}_k \}_{k=1}^n$ satisfies 
    \begin{equation}
        \bm{x}_f = \Phi(t_f,t_0) \bm{x}_0 + \sum_{k=1}^n \Phi(t_f,t_k) B(t_k) \bm{v}_k
    \end{equation}
\end{problem}

\section{Problem Formulation}

The novel contribution of this paper is the inclusion of input magnitude constraints on the control input. The consideration of input magnitude constraints is nontrivial, since bounded control magnitudes, as are typically enforced in aerospace literature for continuous thrust formulations \cite{acikmese,taheri}, prevents the optimal solution from concentrating mass at discrete time points. As will be shown, the formulation presented here maintains impulsivity of the input profile, which provides both algorithmic benefits and a desirable input structure for spacecraft which are often modeled as impulsively actuated.

The following modification to Problem~\ref{prob:BasicProblem} is proposed:
\begin{problem}[Magnitude Constrained Problem]
    Find $\bm{u}(\cdot)$ that solves the primal optimization problem $\mathcal{P}(\{ T_k \}_{k=1}^N) =$
    \begin{equation}
    \begin{aligned}
    \min_{\bm{u}(\cdot)} \ & \sum_{k=1}^N \int_{T_k} f_k (\bm{u}(t)) \text{d} t \\\text{s.t.} \ & \bm{x}_f = \Phi(t_f,t_0) \bm{x}_0 + \sum_{k=1}^N \int_{T_k} \Phi(t_f,t) B(t) \bm{u}(t) \text{d} t, \\
    &\int_{T_k} f_k(\bm{u}(t)) \text{d}t \le \Delta V_k \quad \forall k \in \mathbb{Z}_{1}^N,
    \end{aligned}
    \label{eq:ConstrainedOptimization}
    \end{equation}
    where $f_k(\cdot)$ are norms, $\Delta V_k$ are fixed impulse magnitude limits, and $T_k$ represent ordered control input time windows that satisfy $\bigcup_{k=1}^N T_k \subseteq [t_0,t_f], \ T_i \cap T_j = \emptyset \ \forall i \neq j.$
\label{prob:ConstrainedProblem}
\end{problem}
\subsection{Dual Formulation}
To show that Problem~\ref{prob:ConstrainedProblem} is amenable to being solved through semi-infinite programming, the dual form can be derived in a similar manner to that used for Problem~\ref{prob:BasicProblem}. The Lagrangian of Equation~\ref{eq:ConstrainedOptimization} is
\begin{equation}
\begin{gathered}
    \mathcal{L}(\bm{u}(\cdot),\bm{\lambda}_f, \bm{\sigma} ) =  \bm{\lambda}_f^\intercal (\bm{x}_f - \Phi(t_f,t_0) \bm{x}_0 ) - \sum_{k=1}^N \Delta V_k \sigma_k  \\
    + \sum_{k=1}^N \int_{T_k} (1+\sigma_k) f_k(\bm{u}(t)) - \bm{\lambda}_f^\intercal \Phi(t_f,t) B(t) \bm{u}(t) \text{d} t,
    \end{gathered}
\end{equation}
where $\bm{\lambda}_f$ and $ \bm{\sigma} = (\sigma_k)_{k=1}^N$ are Lagrange multipliers associated with the equality and inequality constraints of the primal in Problem~\ref{prob:ConstrainedProblem} respectively. The associated dual function is given by
\begin{equation}
\begin{gathered}
    g(\bm{\lambda}_f,\{ \sigma_k\}_{k=1}^N) = \inf_{\bm{u}(\cdot)} \mathcal{L}(\bm{u}(\cdot),\bm{\lambda}_f,\{ \sigma_k\}_{k=1}^N) \\
    =
     \begin{cases}
        \bm{\lambda}_f^{\intercal} (\bm{x}_f - \Phi(t_f,t_0) \bm{x}_0 ) - \sum_{k=1}^N \Delta V_k \sigma_k,  \\  \text{if} \ f_k^\circ(B^\intercal(t) \Phi^\intercal(t_f,t) \bm{\lambda}_f ) \le 1+\sigma_k \ \\ \forall t \in T_k, \forall k \in \mathbb{Z}_{1}^N; -\infty, \text{otherwise}
    \end{cases}
    \end{gathered}
\end{equation}
From this, the dual form is arrived to with
\begin{problem}[Magnitude Constrained Dual Problem]
Find $\bm{\lambda}_f, \bm{\sigma}$ such that $\mathcal{D}(\{ T_k \}_{k=1}^N) = $
\begin{equation}
    \begin{aligned}
\sup_{\bm{\lambda}_f, \bm{\sigma}\ge 0} \quad & \bm{\lambda}_f^\intercal (\bm{x}_f - \Phi(t_f,t_0) \bm{x}_0) - \sum_{k=1}^N \Delta V_k \sigma_k \\
\text{s.t.} \quad & f_k^\circ(B^\intercal(t) \Phi^\intercal(t_f,t) \bm{\lambda}_f ) \le 1 + \sigma_k \\ & \forall t \in T_k, \ \forall k \in \mathbb{Z}_{1}^N
\end{aligned}
\label{eq:ConstrainedDual}
\end{equation}
\label{prob:ConstrainedDual}
\end{problem}
Similar to Problem~\ref{prob:BasicDual}, the problem has been reduced from an optimization over an infinite function space in the primal to a convex semi-infinite optimization problem over $n_x + N$ variables. 

\begin{lemma}
    Problem~\ref{prob:ConstrainedDual} is strongly dual to the primal Problem~\ref{prob:ConstrainedProblem}, i.e. $\mathcal{D} = \mathcal{P}$.
    \label{lemma:strongduality}
\end{lemma}
\begin{proof}
    Evidently, Problem~\ref{prob:ConstrainedDual} satisfies Slater's condition for semi-infinite programs (see Definition 2.1 from \cite{shapiro}) since there is a feasible interior point, i.e. $\exists \bm{\lambda}_f, \bm{\sigma} >0, \text{s.t.} f_k^\circ(B^\intercal(t)\Phi(t_f,t) \bm{\lambda}_f ) < 1 +\sigma_k \ \forall t\in T_k, \forall k \in \mathbb{Z}_{1}^N$. From Theorem 2.3 of \cite{shapiro}, this implies that strong duality holds for the semi-infinite program. 
\end{proof}
\begin{lemma}
    The semi-infinite dual Problem~\ref{prob:ConstrainedDual} is reducible \cite{shapiro}, i.e., there exists a finite subset of times $\{t\} \triangleq \{t_1,\dots,t_M\} \subset \bigcup_{k=1}^N T_k$ such that $\mathcal{D}(\{ t \}) = \mathcal{D}(\{ T_k\}_{k=1}^N)$. 
\label{lemma:reducibility}
\end{lemma}
\begin{proof}
    Since the problem exhibits strong duality, the problem is reducible so long as an optimal solution exists (the primal is feasible) according to Theorem 3.1 of \cite{shapiro}.
\end{proof}
Lemmas~\ref{lemma:strongduality} and \ref{lemma:reducibility} extend several important properties to the proposed constrained problem. Namely, unboundedness of the dual problem implies infeasibility of the primal, the optimal solution retains impulsivity (i.e., $\bm{u}(t) = 0 $ a.e.), and exchange methods  \cite{Reemtsen1998} are valid for solving Problem~\ref{prob:ConstrainedDual}.

\subsection{Input Reconstruction}

The technique used to recover the optimal input to the unconstrained problem from the dual formulation can be similarly applied to its constrained version. If $\bm{\lambda}_f^*, \bm{\sigma}^*$ are the optimal solution variables to Problem~\ref{prob:ConstrainedDual}, then the optimal impulse times for a given window are recovered from the set
\begin{equation}
\begin{aligned}
    \{ t\}^*_k &= \{ t\in T_k \mid f_k^\circ(B^\intercal(t) \Phi^\intercal(t_f,t)\bm{\lambda}_f^* ) = 1 + \sigma_k^* \} \\
    &=\{t_{k1},t_{k2}, \dots, t_{k,m_k} \}.
\end{aligned}
\end{equation}
and the corresponding input is given by
\begin{equation}
    \bm{u}^*(t) = \sum_{k=1}^N \sum_{j=1}^{m_k} \bm{v}_{kj} \delta(t-t_{kj}).
\end{equation}
Furthermore, from the complementary slackness condition, it must hold that
\begin{equation}
    \sigma_k \left( \int_{T_k} f_k(\bm{u}(t)) \text{d}t - \Delta V_k  \right) = 0 \ \forall k \in\mathbb{Z}_{1}^N,
\end{equation}
which can be leveraged to show the interesting property that
\begin{equation}
    \sigma_k > 0 \implies \int_{T_k} f_k(\bm{u}(t))\text{d} t = \sum_{j =1}^{m_k} f_k(\bm{v}_{kj}) = \Delta V_k
\end{equation}
meaning that the non-zero elements of $\bm{\sigma}$ correspond to time windows in which the input is saturated. 

From the optimal impulse times, the optimal input can be recovered from the tractable finite LP/SOCP:
\begin{problem}[Constrained Input Reconstruction Problem]
    Solve the discretized primal Problem~\ref{prob:ConstrainedProblem} from the optimal impulse times recovered from Problem~\ref{prob:ConstrainedDual} with $\mathcal{P}_d= \mathcal{P}( \{\{ t \}_k^* \}_{k=1}^N)$.
    \label{prob:DiscretePrimal}
\end{problem}

\begin{theorem}
    A feasible optimal solution to Problem~\ref{prob:ConstrainedProblem} consists of at most $n_x+N$ impulses.
\end{theorem}
\begin{proof}
    Since strong duality holds and the dual form is a convex semi-infinite program of $n_x+N$ variables, Theorem 3.2 of \cite{shapiro} states that there exists a discretization $S$ which satisfies $|S| \le n_x +N$ such that $\mathcal{D}(S) = \mathcal{D}$, leading to at most $n_x +N$ points or atoms in time at which $\bm{u}(t)$ may be nonzero. 
\end{proof}
\begin{remark}
    The number of impulses can be further reduced when considering the addition of the $\bm{\sigma} \ge 0$ constraints, which reduces the active support of the semi-infinite constraint, leading to at most $n_x +N_a$ points in time at which the input may be nonzero, where $N_a$ is the number of strictly positive $\bm{\sigma}$. 
\end{remark}

\section{Algorithmic Approach}
Solving Problem~\ref{prob:ConstrainedProblem} in an algorithmically tractable manner can be done through first solving the dual problem using exchange methods to determine the optimal impulse locations and then reconstructing the input through the discretized primal. 
\subsection{Dual Problem Algorithm}
This work leverages an exchange method \cite{Reemtsen1998} for solving the semi-infinite program, which consists of iteratively solving a finite version of Problem~\ref{prob:ConstrainedDual} and adding constraints corresponding to the worst semi-infinite constraint violation. To illustrate this iterative technique, define the modified form of Problem~\ref{prob:ConstrainedDual}, given by
\begin{problem}
Solve the discretized form of Problem~\ref{prob:ConstrainedDual} $\mathcal{D}(S)$, where $S \triangleq \{ s_1,s_2,\dots s_d\}\subset \bigcup_{k=1}^N T_k$ is a finite set, with the addition of the constraint:
\begin{equation}
    \bm{\lambda}_f^\intercal (\bm{x}_f - \Phi(t_f,t_0) \bm{x}_0) - \sum_{k =1}^N \Delta V_k \sigma_k \le \Delta V_{tot},
\end{equation}
where $\Delta V_{tot} = \sum_{k=1}^N \Delta V_k$ is the total amount of expendable fuel imposed.
\label{prob:DiscretizedConstrainedDual}
\end{problem}
Since Problem~\ref{prob:ConstrainedDual} is reducible according to Lemma~\ref{lemma:reducibility}, then it holds that the finite version of the problem can solve the semi-infinite problem if $S$ is chosen correctly. The additional constraint in Problem~\ref{prob:DiscretizedConstrainedDual} prevents artificial infeasibility in the case that the discrete form is under constrained in the exchange method process. This constraint is natural to apply since $\mathcal{D} \le \mathcal{P} \le \Delta V_{tot}$ if a feasible point exists, meaning there is no loss in feasibility when this constraint is imposed. 

To leverage the reducibility property of the dual, exchange methods gradually add values to the set $S$ such that $\mathcal{D}(S^{(j)}) \to \mathcal{D}$ \cite{Reemtsen1998,Hettich2009}. Several prominent techniques exist in exchange methods for augmenting the discrete set $S$, which effect problem size growth and convergence rate. The technique proposed here simply evaluates the semi-infinite constraint over a grid and adds time points corresponding to the largest constraint violations.
\begin{equation}
\begin{gathered}
    S^{(j+1)} \gets S^{(j)} \cup  \arglocmax_{\scriptscriptstyle{t \in\cup_{k=1}^N T_k:f_k^\circ(\bm{\lambda}_v^{(j)}(t)) > 1+\sigma_k^{(j)}} }f_k^\circ(\bm{\lambda}_v^{(j)}(t))
    \end{gathered}
    \label{eq:ExchangeMethod}
\end{equation}
Here, $\bm{\lambda}_v^{(j)}(t) = B^\intercal(t) \Phi^\intercal(t_f,t) \bm{\lambda}_f^{(j)} $ is the primer vector at iteration $j$ of the algorithm. Since $f_k^\circ(\bm{\lambda}_v^{(j)}(t))$ is trivial to compute in this case, this process does not introduce significant computational burden in the algorithmic framework. The discrete set can be initialized with simply the initial and final times, i.e., $S^{(1)} = \{t_0,t_f\}$ or with the optimal times associated with a two-impulse burn as is done in \cite{ARZELIER2016373}. The selection of the initial times and the addition of new times are hyperparameters of the algorithmic framework which can be adjusted for specific problems, akin to that done in \cite{koenig2019}. 

These algorithmic modifications retain several important features of exchange methods. Namely, it is evident that $\mathcal{D} \le\mathcal{D}(S^{(j+1)}) \le \mathcal{D}(S^{(j)})$ since $S^{(j)} \subseteq S^{(j+1)}$. Moreover, given a tolerance $\epsilon$, an important result is as follows, which is extended from \cite{ARZELIER2016373}.
\begin{lemma}
    Suppose Problem~\ref{prob:DiscretizedConstrainedDual} is solved such that $f_k^\circ(B^\intercal(t) \Phi^\intercal(t_f,t) \bm{\lambda}_f^* ) - \sigma_k^* \le 1+\epsilon \ \forall t \in T_k, \forall k \in \mathbb{Z}_1^N$. Then it follows that $\mathcal{D}(S^{(j)}) \le (1+\epsilon) \mathcal{D}$.
    \label{lemma:epsilon}
\end{lemma}
\begin{proof}
    Let $ \bar{\bm{\lambda}}_f = \bm{\lambda}_f^*/(1+\epsilon)$ and $\bar{\sigma}_k = \sigma_k^*/(1+\epsilon)$. Evidently, $(\bar{\bm{\lambda}}_f,\bar{\bm{\sigma}})$ is a feasible point to Problem~\ref{prob:ConstrainedDual} that satisfies $\mathcal{D}(S^{(j)})/(1+\epsilon) = \bar{\bm{\lambda}}_f^\intercal (\bm{x}_f - \Phi(t_f,t_0) \bm{x}_0 ) - \sum_{k=1}^N \Delta V_k \bar{\sigma}_k \le \mathcal{D} \implies \mathcal{D}(S^{(j)}) \le (1+\epsilon) \mathcal{D}$.
\end{proof}
Lemma~\ref{lemma:epsilon} can be exploited in an algorithmic framework to ensure that the dual problem is solved to within a specified suboptimality tolerance. Algorithm~\ref{alg:Dual} summarizes the procedure for solving the dual problem.
\vspace*{-4pt}
\begin{algorithm}[H]
\begin{algorithmic}[1]
\REQUIRE{LTV Dynamics $\{\Phi(\tau_f,\tau_0), A(\tau), B(\tau) \}$, Time Intervals $\{ T_k \}_{k=1}^N$, Norms $\{ f_k(\cdot) \}_{k=1}^N$ Boundary Conditions $\{\bm{x}_0,\bm{x}_f\}$, Convergence Tolerance $\epsilon$}
\ENSURE{Optimal Cost $\mathcal{D}(S^{(j)})$, Optimal Values $\bm{\lambda}_f,\bm{\sigma}$}
\STATE{Discretize $\{T_k\}_{k=1}^N$, $S^{(1)} \gets \{ t_0, t_f \}$, $j \gets 1$}
\WHILE{$\exists k\in\mathbb{Z}_1^N,t\in T_k \ \text{s.t.} \ f^\circ_k(B^\intercal(t) \Phi^\intercal(t_f,t)\bm{\lambda}_f^{(j)} ) -\sigma_k^{(j)} > 1+\epsilon$}
\STATE{$\bm{\lambda}_f^{(j)},\bm{\sigma}^{(j)}, \mathcal{D}(S^{(j)}) \gets$Solve Problem~\ref{prob:DiscretizedConstrainedDual} with $S^{(j)}$}
\STATE{$S^{(j+1)} \gets S^{(j)}\cup \Delta S$ (Equation~\ref{eq:ExchangeMethod}), $j \gets j + 1$}
\ENDWHILE
\IF{$\mathcal{D}(S^{(j)}) = \Delta V_{tot}$} 
\RETURN{Problem Infeasible}
\ENDIF
\vspace{-.4cm}
\label{alg:Dual}
\end{algorithmic}
\caption{Dual Semi-Infinite Program Algorithm.}
\label{alg:Dual}
\end{algorithm}

\subsection{Input Reconstruction}
Reconstructing the input signal from the dual solution requires extracting a discrete set of times where the semi-infinite constraints are active and solving the discretized form of the primal problem. These aspects are summarized in Algorithm~\ref{alg:Primal}.
\vspace{-.2cm}
\begin{algorithm}[H]
\begin{algorithmic}[1]
\REQUIRE{LTV Dynamics $\{\Phi(\tau_f,\tau_0), A(\tau), B(\tau) \}$, Time Intervals $\{ T_k \}_{k=1}^N$, Norms $\{ f_k(\cdot) \}_{k=1}^N$ Boundary Conditions $\{\bm{x}_0,\bm{x}_f\}$, Dual Solution $\{\bm{\lambda}_f^*,\bm{\sigma}^* \}$}
\ENSURE{Optimal Cost $\mathcal{P}_d$, Input Signal $\bm{u}(t)$, State $\bm{x}(t)$}
\STATE{$\{ t\}^*_k = \{ t\in T_k \mid f_k^\circ(B^\intercal(t) \Phi^\intercal(t_f,t)\bm{\lambda}_f^* ) - (1 + \sigma_k^*) \in [-\epsilon,\epsilon] \}$}
\STATE{$\{\{\bm{v}_{kj}\}_{j=1}^{m_k}\}_{k=1}^N, \mathcal{P}_d\gets$ Solution to Problem~\ref{prob:DiscretePrimal}}
\STATE{$\bm{u}(t) \gets \sum_{k=1}^N \sum_{j=1}^{m_k} \bm{v}_{kj} \delta(t-t_{kj})$}
\STATE{$\bm{x}(t) \gets \Phi(t,t_0) \bm{x}_0 + \int_{t_0}^{t} \Phi(t,\tau) B(\tau) u(\tau) \text{d} \tau$}
\vspace{-.6cm}
\end{algorithmic}
\caption{Input Reconstruction Algorithm.}
\label{alg:Primal}
\end{algorithm}
Though the above procedure is the most general, it is possible to exploit further algorithmic performance in the case of smooth norm costs and constraints, in which case the direction of the input is known and the ambiguous input magnitudes can be solved through least squares.
\begin{theorem}
    Solving the infinite program Problem~\ref{prob:ConstrainedProblem}, by means of Algorithm~\ref{alg:Dual} followed by Algorithm~\ref{alg:Primal}, is globally convergent to the global minimum to within the imposed tolerance $\epsilon$.
\end{theorem}
\begin{proof}
    Since Problem~\ref{prob:DiscretizedConstrainedDual} is convex and bounded above, each subproblem of Algorithm~\ref{alg:Dual} will converge. Moreover, Algorithm~\ref{alg:Dual} will always terminate since $|S^{(j)}|$ is bounded. Finally, since $\mathcal{D}(S^*) \le \mathcal{D}(1+\epsilon)$ (Lemma~\ref{lemma:epsilon}), $\mathcal{D} = \mathcal{P}$ (Lemma~\ref{lemma:strongduality}), and $\mathcal{D}(S^*) = \mathcal{P}(S^*)$ (strong duality of SOCPs/LPs), $\mathcal{P}(S^*) \le \mathcal{P}(1+\epsilon)$.
\end{proof}

\section{Performance and Validation}

The problem formulation and corresponding algorithmic approach is validated on the problem of impulsive control for proximity operations of satellites. 
\subsection{System Definition}
Relative motion of satellites in orbits subject to single body gravitational force can be modeled as linear time-varying systems under a variety of perturbations, such as drag or spherical harmonics, through a relative orbital element state representation \cite{damico,koenigdynamics}, defined by
\begin{equation}
    \delta \textit{\textbf{\oe}} = \begin{bmatrix}
        \delta a & \delta \lambda & \delta e_x & \delta e_y & \delta i_x & \delta i_y
    \end{bmatrix}^\intercal.
\end{equation}
These states are functions of the classical Keplerian orbital elements of circular two-body motion of the target and chaser spacecraft. The linearized equations of motion under this state are of Jordan form and can be found in \cite{koenigdynamics} along with transformation matrices to RTN frame cartesian state. The ensuing examples use an orbital mean motion of $n=.00113$, consistent with LEO.

\subsection{Planar Reconfiguration Demonstration}

The first example considers a planar rendezvous problem consisting of reducing along track separation from a distance of 100 km to 10 km, i.e., $a\delta \textit{\textbf{\oe}}_0 = [10,100,0,10,0,0]^\intercal$ [km] to $a\delta \textit{\textbf{\oe}}_f = [0,10,0,0,0,0]^\intercal$ [km]. The transfer occurs over four orbits, with eight equally spaced time windows enforced over the duration of the transfer with a buffer of $1/10$th of an orbit between each window. Each time window is discretized over 1000 time points in the dual optimization scheme to provide sufficient resolution. The input is scaled such that a unitary impulse corresponds to $1.13$ [m/s] of $\Delta V$ expenditure. Moreover, the norm applied is the $l_2$ norm, such that $f_k(\cdot) = \| \cdot \|_2 \ \forall k$.

\begin{figure}[!t]
\vspace*{5pt}
\centering
\begin{minipage}[t]{0.228\linewidth}
  \vspace{0pt}
  \centering
  \subfloat{%
    \includegraphics[width=\linewidth]{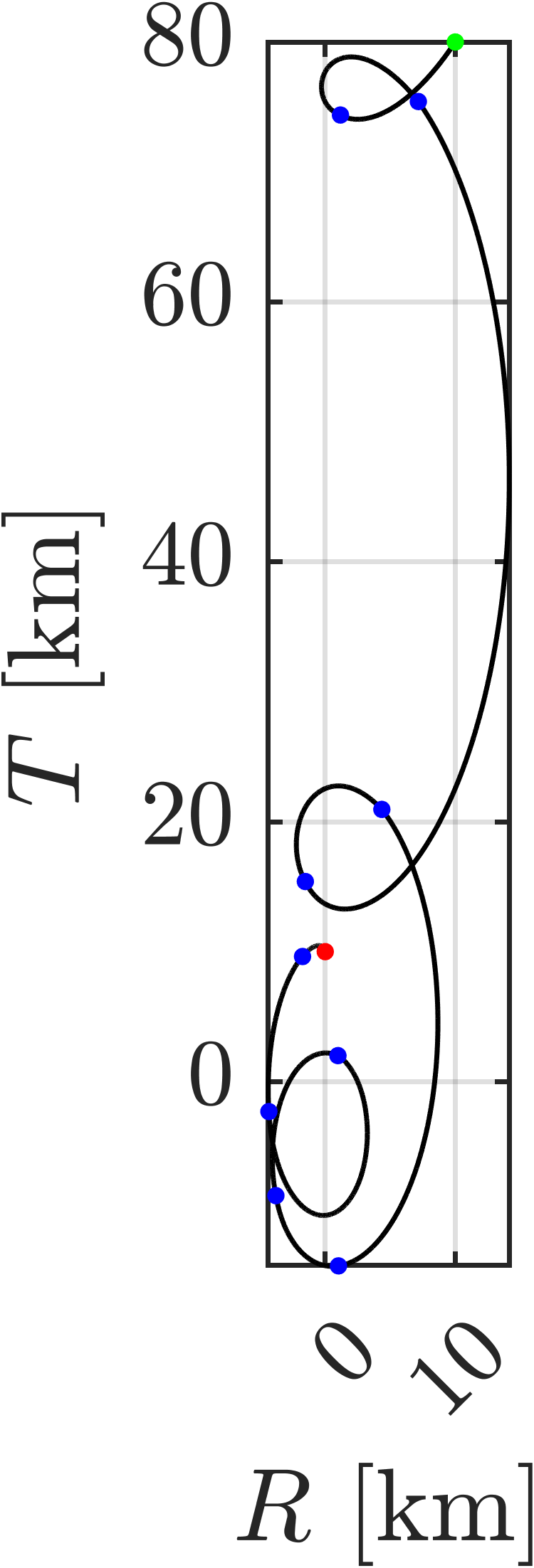}%
    \label{fig:PlanarReconfigurationRelativePosition}}%
\end{minipage}%
\hfill
\begin{minipage}[t]{0.17\linewidth}
  \vspace{3pt}
  \centering
  \subfloat{%
    \includegraphics[width=\linewidth]{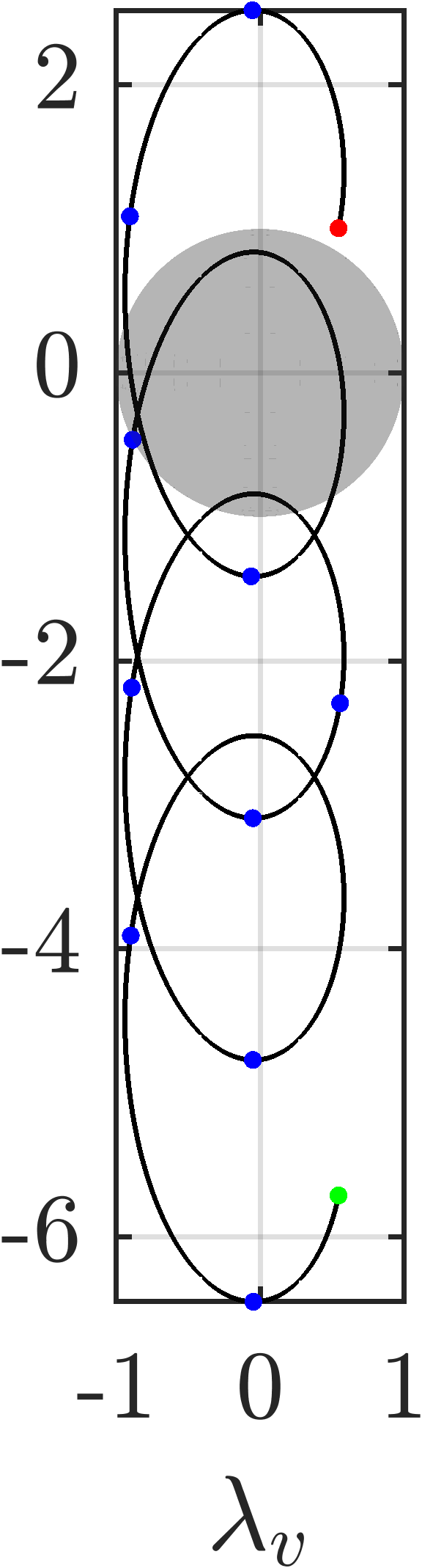}%
    \label{fig:PlanarReconfigurationLambda}}%
\end{minipage}%
\hfill
\begin{minipage}[t]{0.55\linewidth}
  \vspace{0pt}
  \subfloat{%
    \includegraphics[width=\linewidth]{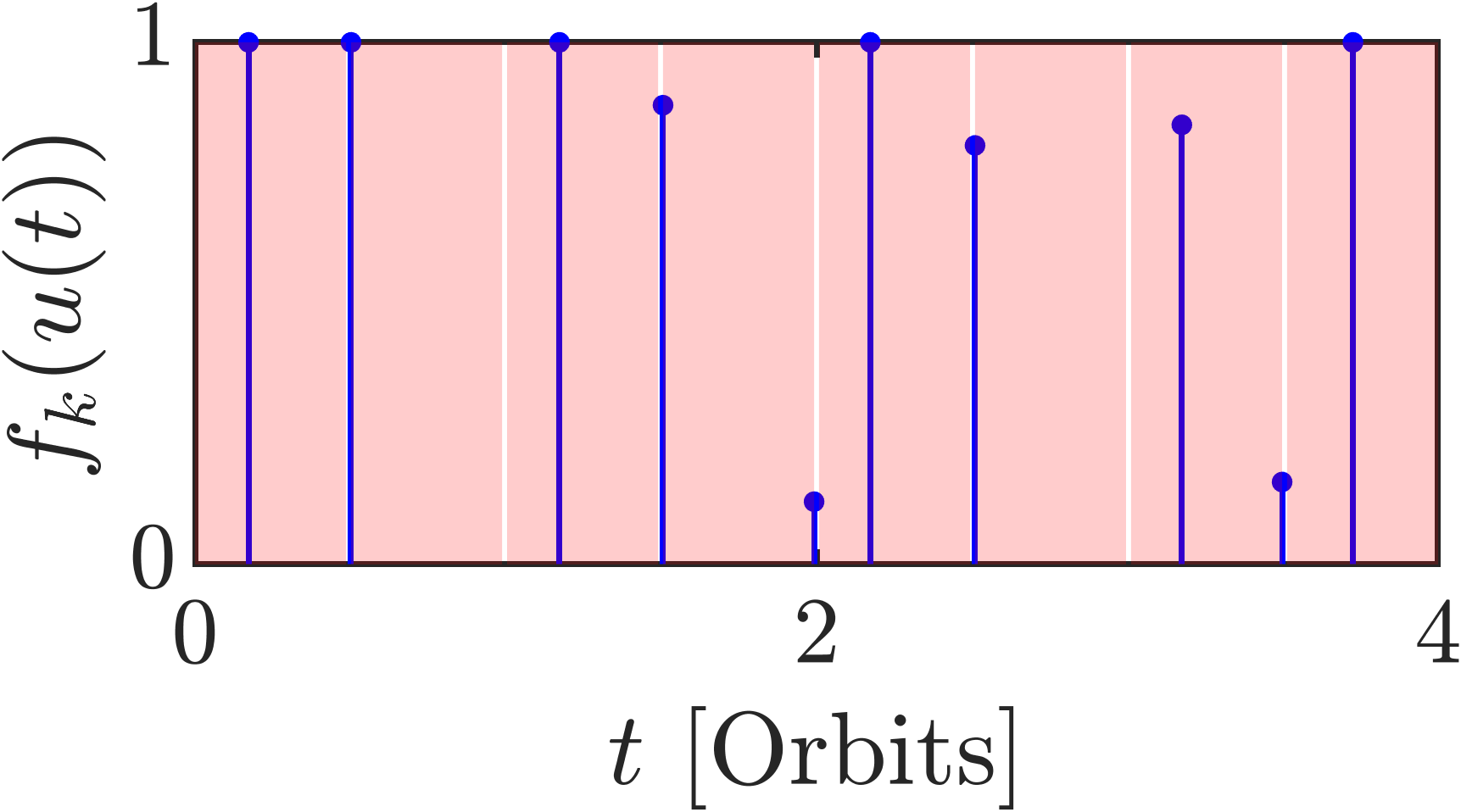}%
    \label{fig:PlanarReconfigurationThrottle}}\\
  \subfloat{%
    \includegraphics[width=\linewidth]{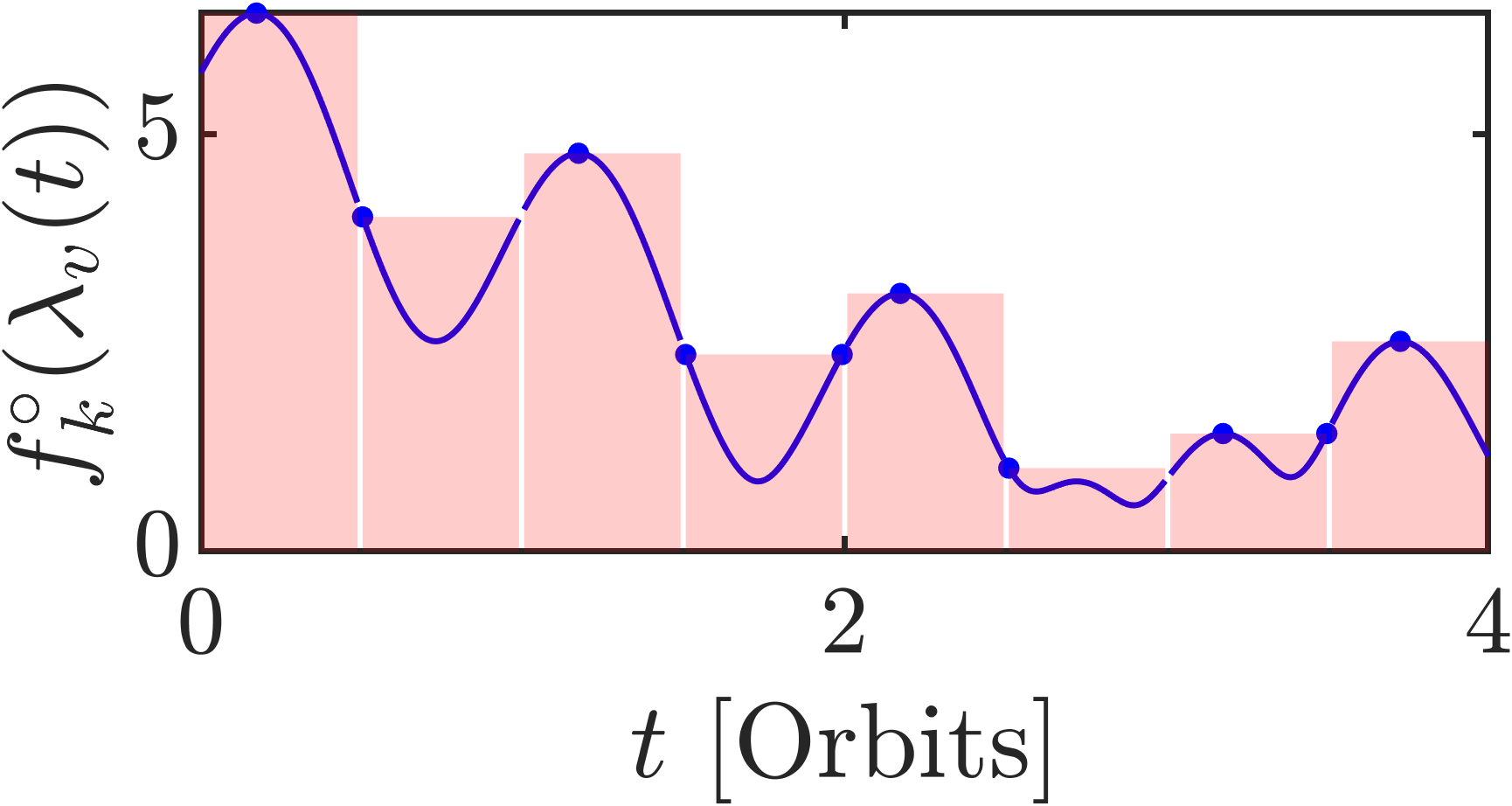}%
    \label{fig:PlanarReconfigurationDualNorm}}%
\end{minipage}

\caption{Planar Reconfiguration Trajectory.}
\label{fig:PlanarReconfiguration}
\vspace{-.6cm}
\end{figure}

\begin{figure}
    \centering
    \subfloat{%
    \includegraphics[height=3cm]{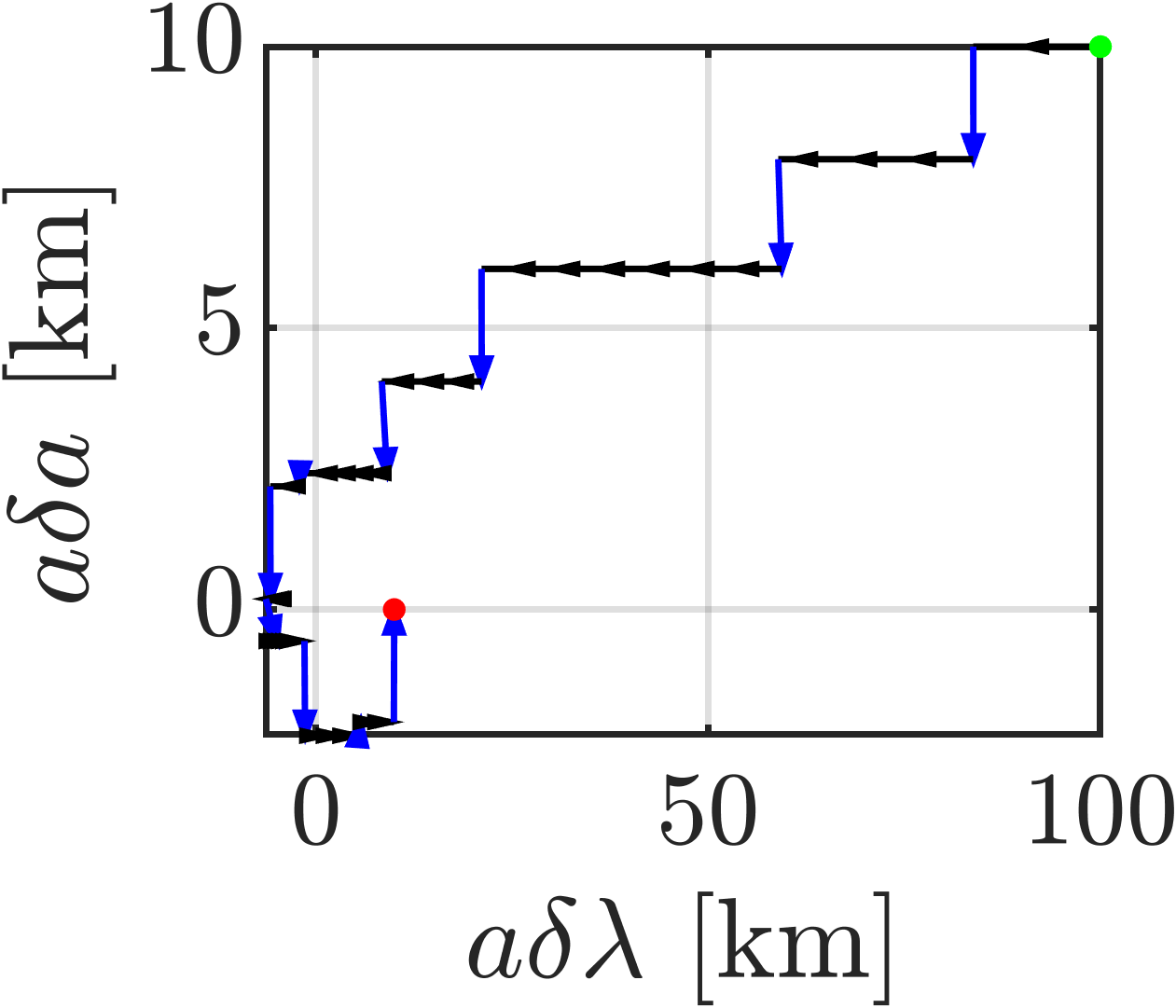}%
    \label{fig:PlanarReconfigurationda}}
    \subfloat{%
    \includegraphics[height=3cm]{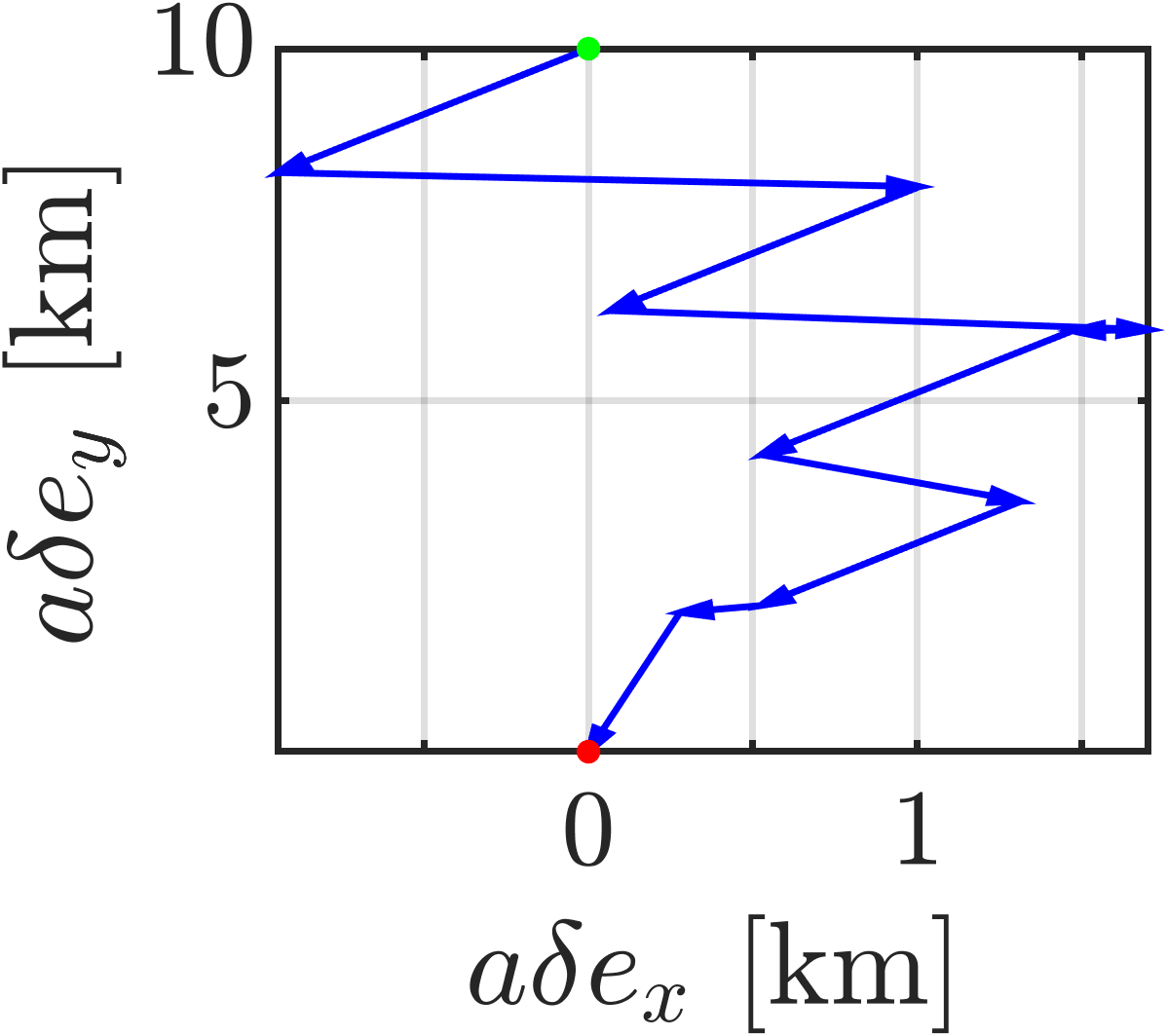}%
    \label{fig:PlanarReconfigurationde}}
    \vspace{-.2cm}
    \caption{Planar Reconfiguration Relative Orbital Elements}
    \label{fig:PlanarReconfigurationOEs}
    \vspace{-.6cm}
\end{figure}
Figure~\ref{fig:PlanarReconfiguration} displays the RT-plane relative position, primer vector trajectory, input impulse profile, and dual norm trajectory. Figure~\ref{fig:PlanarReconfigurationOEs} shows the trajectory from the perspective of the planar relative orbital elements. Blue dots/arrows denote impulse events, green dots denote starting time, and red dots denote terminal time. The computed dual optimal solution for the problem is $\bm{\lambda}_f^* = [0.7380,0.0891,-0.2359,-0.7236,0,0]^\intercal, \bm{\sigma}^* = [5.4503,3.0092,3.7714,1.3608,2.0926,0,0.4142,1.5183]^\intercal$. The total $\Delta V$ expenditure is $8.860 $ [m/s]. Notably, the dual norm exceeds the value of 1 at all but one of the time windows, corresponding exactly to the time windows at which the input constraint is active.

To validate and compare the performance of the proposed algorithm, the problem is also solved under a direct optimization framework whereby the primal problem is discretized with the inputs at each discrete node as solution variables. Under this approach, the number of solution variables increases drastically compared to the dual solution framework, with the total number of solution variables amounting to $n_u N M$, with $n_u$ being the input dimension, $N$ the number of time windows, and $M$ the discretization count per time window. In contrast, the dual SIP involves $n_x + N$ solution variables, with cone count being on the order of $\sim2N$-$3N$, depending on the problem structure. The poor scalability of direct primal optimization manifests as longer solver times, with this issue being especially apparent for long duration transfers, whereby the discretization induces substantial variable counts. 

Under this direct optimization approach, the equivalent problem is solved and produces an equivalent solution. However, the solver runtime is 0.6 seconds, whereas the cumulative solver time for the dual SIP with input reconstruction is 0.1262 seconds\footnote{Runtime experiments were performed on an Intel i7‑10750H CPU in MATLAB R2024a with the MOSEK solver \cite{mosek} and YALMIP modeler \cite{Lofberg2004}. Solver times are determined from extracting the MOSEK solver time from the YALMIP solution. }. 

\subsection{VISORS Transfer Demonstration and Comparison}

The VISORS mission \cite{guffanti2023autonomousguidancenavigationcontrol} is a two-nanosatellite formation-flying mission that aims to perform telescopic measurements of the solar corona. Part of the mission operations involves performing a long duration, impulse constrained transfer from a 200 meter radius standby orbit to a 40 meter science orbit, i.e., $a\delta \textit{\textbf{\oe}}_0 = [0,0,0,200,0,200]^\intercal$ [m] to $a\delta \textit{\textbf{\oe}}_f = [-2.62, 45.21, -34.51, 4.78, -18.72, 2.72]^\intercal$ [m]. The challenge of this transfer is compounded by the limitations of the propulsion systems aboard the nanosatellites, which can provide a maximum impulse of just 2 millimeters per second with at least 45 seconds of cool down time between each impulse. 

To demonstrate the strength of the proposed algorithm for this difficult problem, the standby to science scenario from \cite{guffanti2023autonomousguidancenavigationcontrol} is demonstrated with a transfer duration of 10 orbits and 500 equally spaced time windows are enforced with a buffer of 45 seconds and a maximum impulse expenditure of $2$ millimeters per second per window. The $l_1$ norm is applied, $f_k(\cdot) = \| \cdot \|_1 \ \forall k$, to be consistent with the fuel metric in \cite{guffanti2023autonomousguidancenavigationcontrol}. Each time window is discretized over 100 time points.


\begin{figure}[!t]
\vspace*{5pt}
\centering
\begin{minipage}[t]{0.55\linewidth}
  \vspace{0pt}
  \centering
  \subfloat{%
    \includegraphics[width=\linewidth]{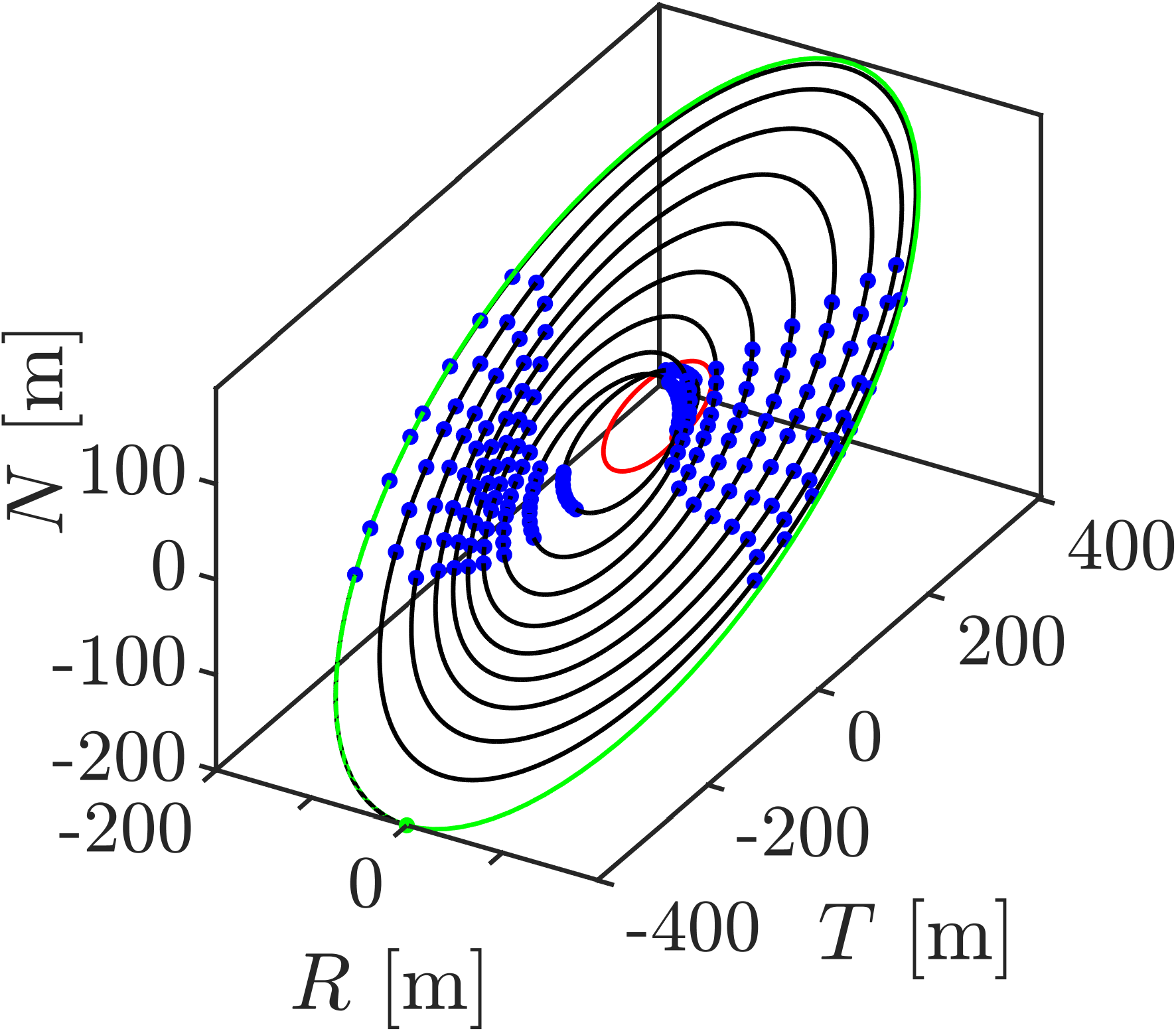}%
    \label{fig:VISORSTrajectory}}%
\end{minipage}%
\hfill
\begin{minipage}[t]{0.44\linewidth}
  \vspace{0pt}
  \raggedleft
  \subfloat{%
    \includegraphics[width=1\linewidth]{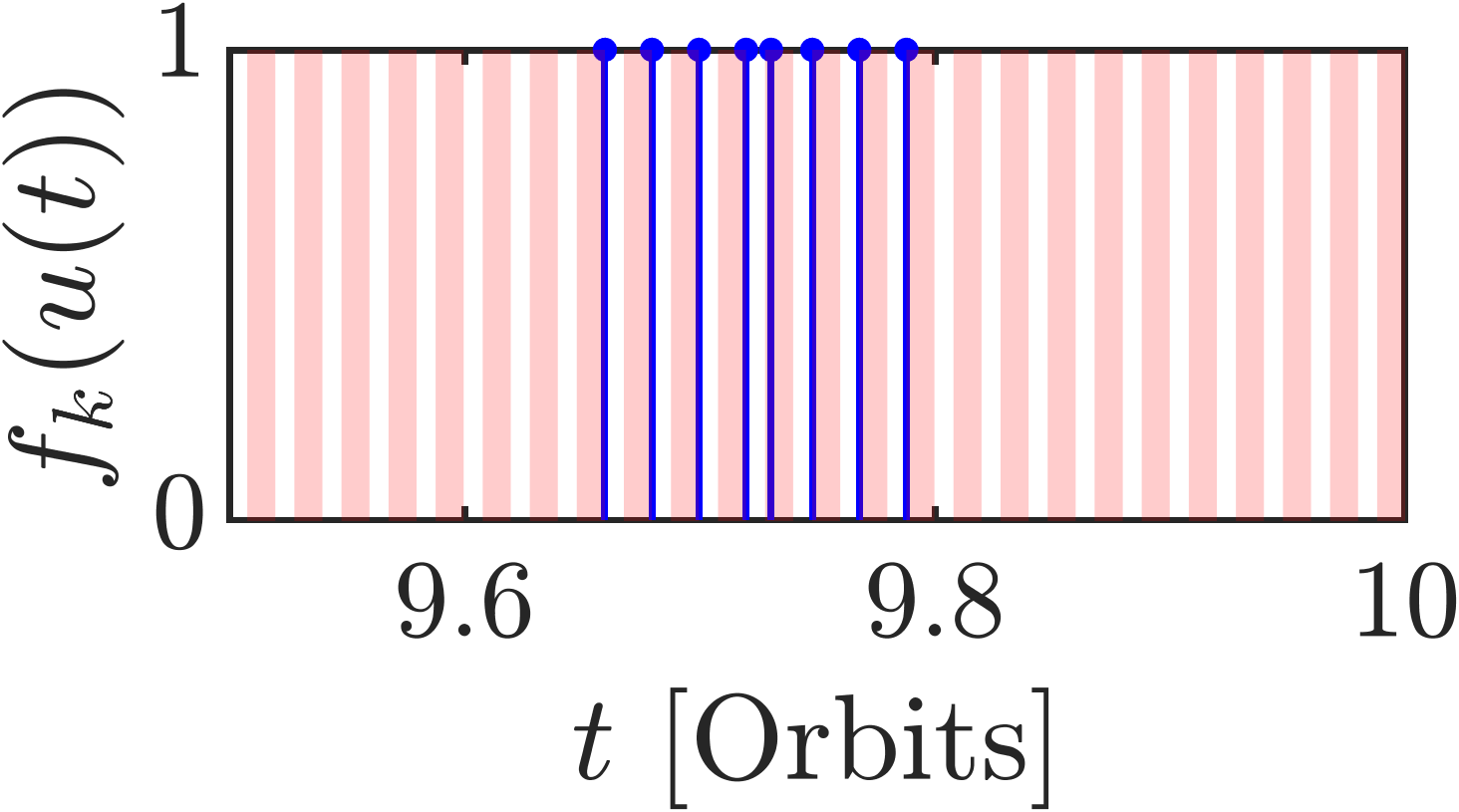}%
    \label{fig:VISORSReconfigurationThrottle}}\\
    \vspace{-8pt}
  \subfloat{%
    \includegraphics[width=1\linewidth]{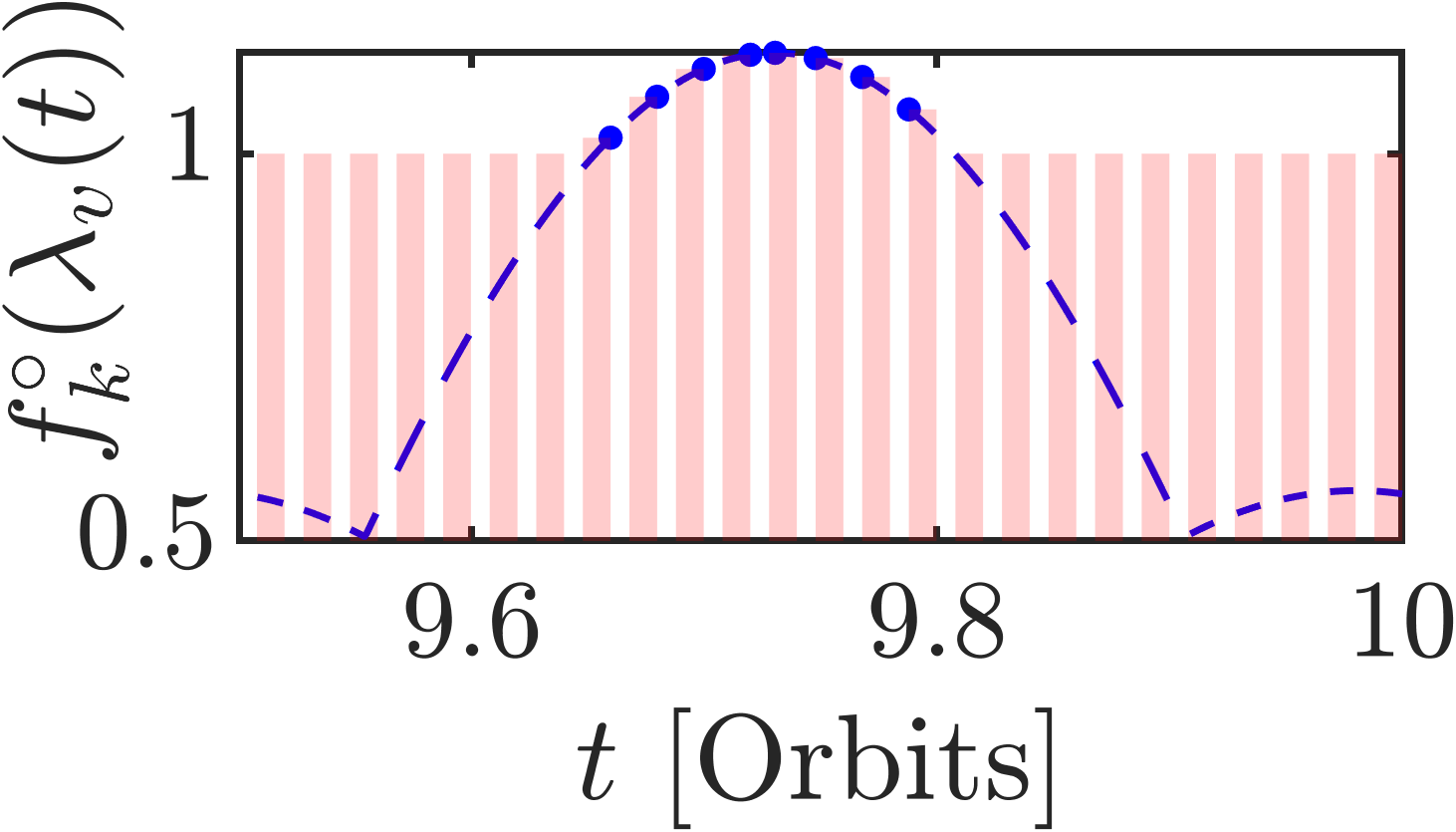}%
    \label{fig:VISORSReconfigurationDualNorm}}%
\end{minipage}

\caption{VISORS Transfer Trajectory.}
\label{fig:VISORS}
\vspace{-.6cm}
\end{figure}


\begin{figure}
    \centering
    \subfloat{%
    \includegraphics[height=2.94cm]{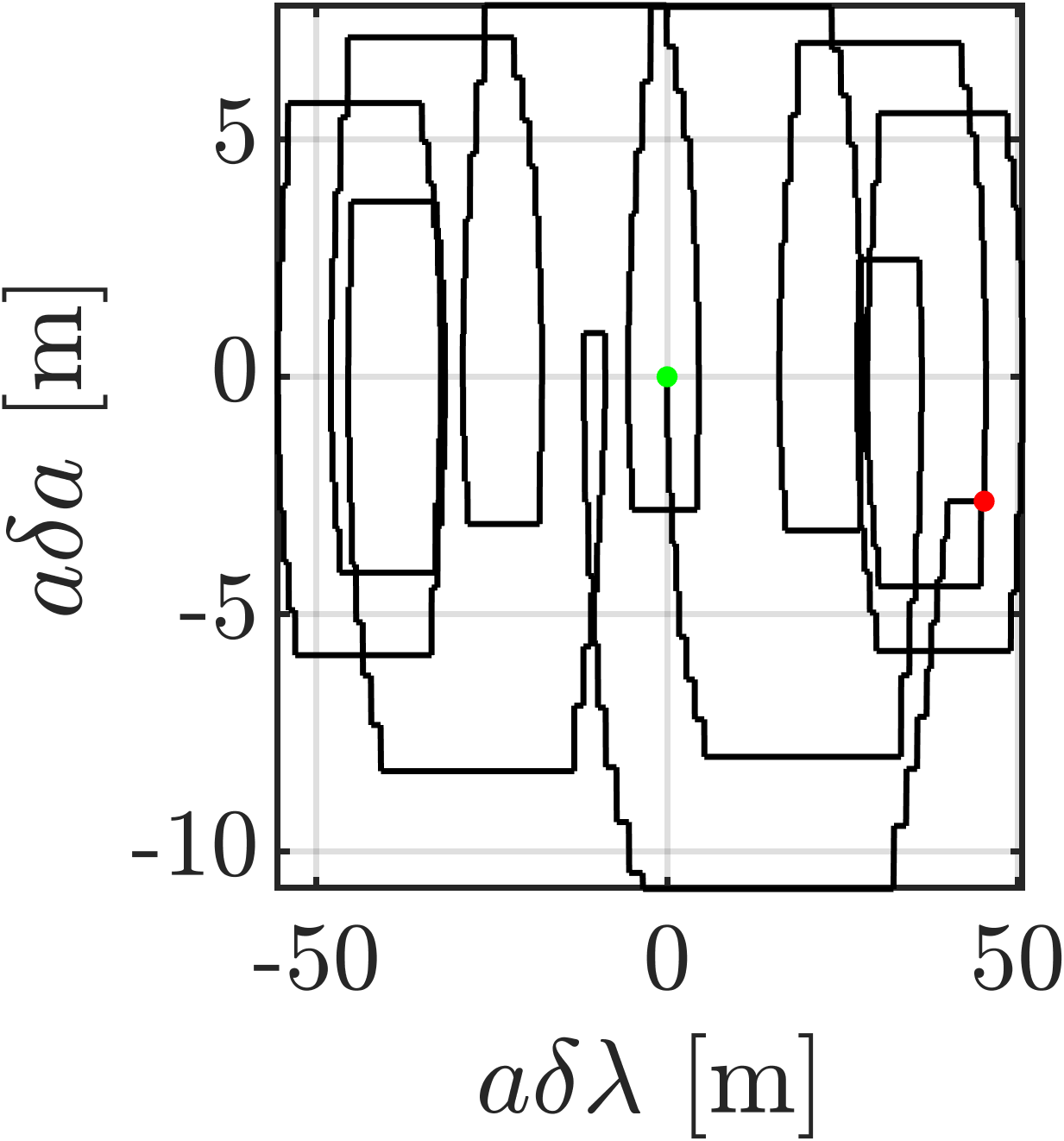}%
    \label{fig:VISORSda}}
    \hfill
    \subfloat{%
    \includegraphics[height=3cm]{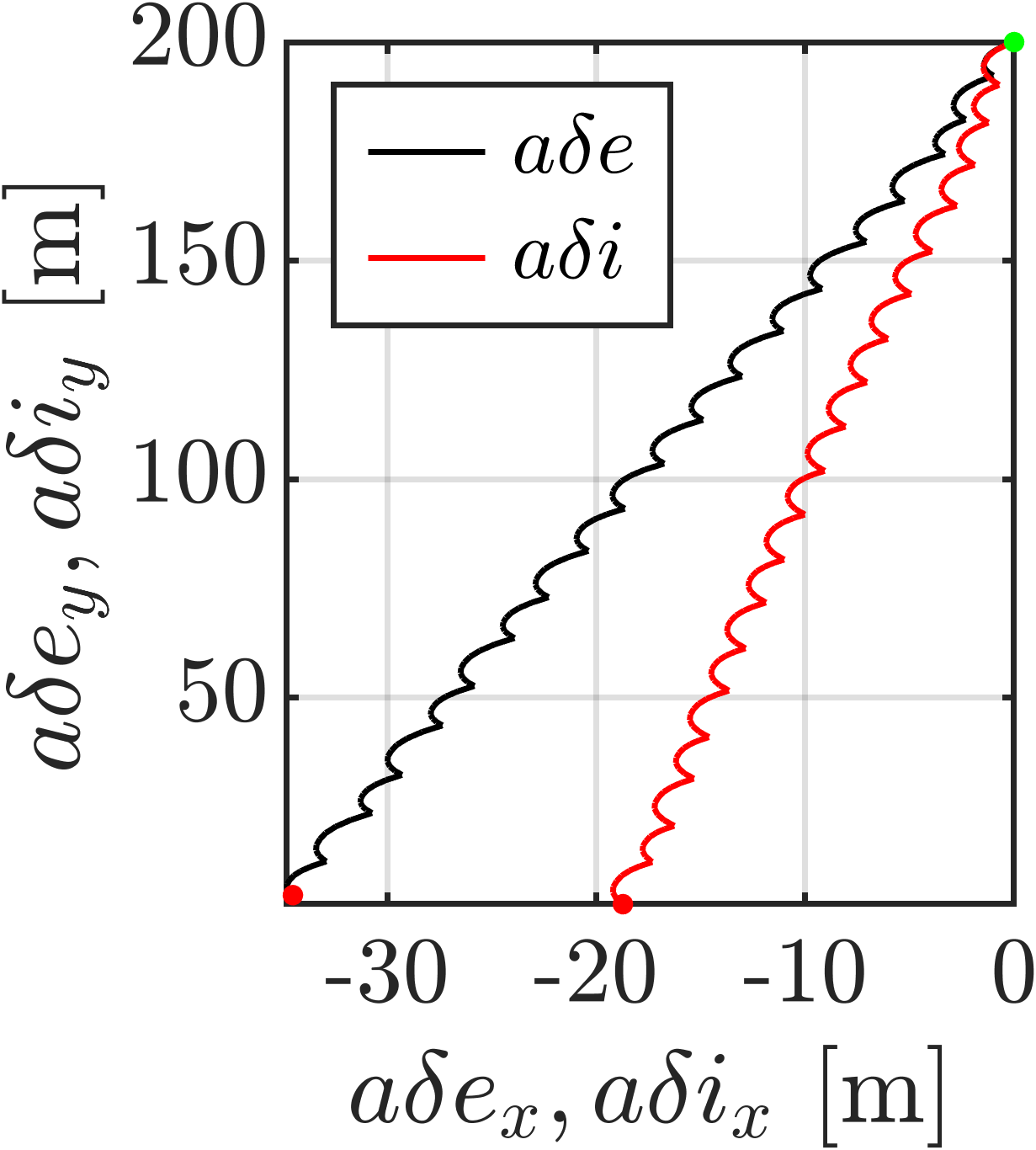}%
    \label{fig:VISORSde}}
    \hfill
    \subfloat{%
    \includegraphics[height=3cm]{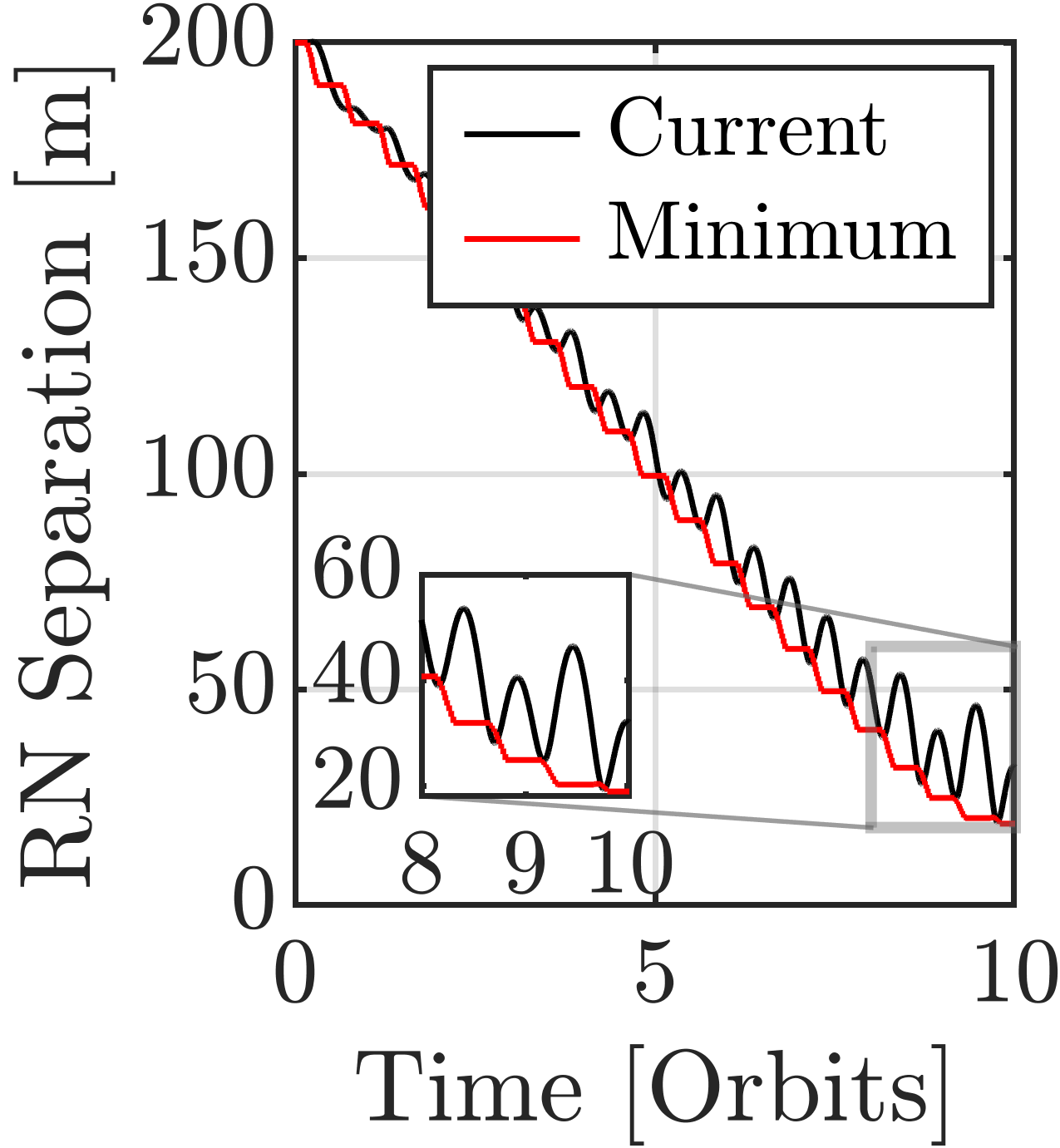}%
    \label{fig:VISORSRN}}
    \vspace{-.2cm}
    \caption{VISORS Transfer Relative Orbital Elements}
    \label{fig:VISORSROEs}
\end{figure}

The computed VISORS transfer trajectory is displayed in Figure~\ref{fig:VISORS}, including the RTN trajectory, a subset of the impulse profile, and a subset of the dual norm trajectory. Additionally, the relative orbital element trajectory is plotted in Figure~\ref{fig:VISORSROEs}, along with the RN plane separation, which validates safe spacecraft separation over the transfer duration \cite{DamicoMontenbruck}.

Finally, this scheme is embedded in a high-fidelity simulation, with the inclusion of execution errors, navigation uncertainty, and dynamical effects that are consistent with \cite{guffanti2023autonomousguidancenavigationcontrol}. The control dynamical model includes drag and $J_2$ effects, as derived in \cite{koenigdynamics}, and the transfer is replanned every orbit for the first 9 orbits and every quarter orbit during the final orbit. Table~\ref{tab:VISORSMC} summarizes the performance over 250 Monte Carlo trials. Importantly, the proposed approach results in an average of $33.27$ [cm/s] of $\Delta V$ expenditure for the transfer, which substantially improves on the $49.34$ [cm/s] metric reported in \cite{guffanti2023autonomousguidancenavigationcontrol}. Given that this transfer is performed 20 times over the mission span, this amounts to an estimated total $\Delta V$ saving of 3.19 meters per second, which is $\sim$25\% of the total $\Delta V$ budget of the VISORS CubeSats.
\begin{table}[!t]
\renewcommand{\arraystretch}{1.3} 
\caption{VISORS Transfer Performance over 250 Monte Carlo Trials with $1\sigma$ Bounds}
\label{tab:VISORSMC}
\centering
\begin{tabular}{ccccccc}
\hline
\hline
    Total $\Delta V$ & Proposed Solver Time & Direct Solver Time \\ \hline
    $33.27 \pm 0.72$ [cm/s] & $0.1912 \pm 0.0480$ [s] & $0.4234 \pm 0.0263$ [s] \\
         \hline
    \end{tabular}
    \vspace{-.5cm}
\end{table}

\section{Conclusion}

This paper demonstrated an extension to a class of moment problems that includes constraints on the integral of the norm of the input signal over time windows of interest. It was shown that the resulting optimal solution remains impulsive and that semi-infinite programming can be applied to obtain optimal solutions with guaranteed convergence. This algorithmic framework was applied to challenging spacecraft relative motion control problems, where it was used to compute maneuver sequences that outperformed established techniques in terms of optimality and runtime. 


\bstctlcite{BSTcontrol}
\bibliographystyle{IEEEtran}
\bibliography{reference}

\end{document}